\documentclass{math-note}

\usepackage[english]{babel}
\usepackage{mathtools,amssymb,amsthm}
\usepackage{needspace}
\usepackage[
  hidelinks,
  pdftitle={Vector Balancing via Directional Total Variation},
  pdfauthor={Shengtao Guo, Ethan X. Fang, Junwei Lu},
  pdfsubject={Directional total variation and the Komlos signing problem}
]{hyperref}

\title{Vector Balancing via Directional Total Variation}
\date{}
\author{
  Shengtao Guo \qquad
  Ethan X. Fang \qquad
  Junwei Lu\thanks{Department of Biostatistics, Harvard T.H. Chan School of
  Public Health. Email: \texttt{junweilu@hsph.harvard.edu}.}
}

\numberwithin{equation}{section}
\newtheorem{theorem}{Theorem}[section]
\newtheorem{proposition}[theorem]{Proposition}
\newtheorem{lemma}[theorem]{Lemma}
\newtheorem{corollary}[theorem]{Corollary}
\theoremstyle{remark}
\newtheorem{remark}[theorem]{Remark}

\newcommand{\R}{\mathbb R}
\newcommand{\Tv}{\mathcal T}
\newcommand{\Prob}{\mathcal P}
\newcommand{\disc}{\operatorname{disc}}
\newcommand{\dd}{\,\mathrm d}

\begin{document}
\maketitle

\begin{abstract}
Our main result is a \(3\sqrt{2\pi}\) bound for the Koml\'os signing
problem: every finite family of real vectors of Euclidean norm at
most one admits a signed sum of \(\ell_\infty\)-norm less than this constant,
independently of the dimension and the family size.
For any \(\kappa\ge0\), if a bounded open convex set supports a probability
density with directional total variation at most \(\kappa\) in every
unit direction, then its
open-set Banaszczyk transform supports another such density with the same
\(\kappa\), provided the translation vector \(v\) satisfies
\(\kappa\|v\|_2\le1/3\).
As a consequence, every finite set system in which each element belongs
to at most \(t\) sets, where \(t\ge1\) is an integer, admits a two-coloring
whose imbalance in each set is less than \(3\sqrt{2\pi t}\).
This gives the square-root dependence predicted by the Beck--Fiala conjecture.
The proof was discovered by the Odin Automatic AI Research Agent.
\end{abstract}

\section{Introduction}

Discrepancy theory asks how evenly a collection of objects can be
two-colored while balancing many tests at once. For a set system on
\(\{1,\ldots,n\}\) with incidence matrix \(A\in\{0,1\}^{m\times n}\),
a coloring \(\varepsilon\in\{-1,1\}^n\) gives the signed set imbalances
as the entries of \(A\varepsilon\).
Allowing real entries gives weighted tests and the definition
\[
 \disc(A)=\min_{\varepsilon\in\{-1,1\}^n}\|A\varepsilon\|_\infty,
 \qquad A\in\R^{m\times n}.
\]
Up to a factor of two, discrepancy measures the error in simultaneously
rounding the all-\(1/2\) vector. If
\(\mathbf1_n\in\R^n\) is the all-one vector and
\(x=(\mathbf1_n+\varepsilon)/2\), then
\(A(x-\tfrac12\mathbf1_n)=\tfrac12A\varepsilon\).
See Bansal, Dadush, and Garg~\cite[Section~1]{BDG19} for this
connection with optimization.

The Koml\'os conjecture asks whether \(\disc(A)\) is bounded by a universal
constant whenever every column of \(A\) has Euclidean norm at most one.
This hypothesis bounds each vector's total squared contribution to the
tests, with no restriction on the number of tests or vectors.
Our main result is the following bound.

\begin{theorem}\label{thm:komlos}
For every pair of positive integers \(m,n\) and every
\(v_1,\ldots,v_n\in\R^m\) with \(\|v_j\|_2\le1\), there are
\(\varepsilon_1,\ldots,\varepsilon_n\in\{-1,1\}\) such that
\begin{equation}\label{eq:main}
 \left\|\sum_{j=1}^n\varepsilon_jv_j\right\|_\infty
 <3\sqrt{2\pi}.
\end{equation}
\end{theorem}

The Euclidean normalization includes bounded-degree set systems.
Its consequence is the square-root dependence in the Beck--Fiala
conjecture~\cite{BF81}; the formulation recalled by Bansal, Dadush,
and Garg~\cite[Section~1]{BDG19} uses the maximum element degree.

\begin{samepage}
\begin{corollary}[Beck--Fiala bound]\label{cor:beck-fiala}
Let \(m,n,t\) be positive integers. If \(A\in\{0,1\}^{m\times n}\) has
at most \(t\) nonzero entries in each column, then
\[
 \disc(A)<3\sqrt{2\pi t}.
\]
\end{corollary}
\begin{proof}
Every column of \(A/\sqrt t\) has Euclidean norm at most one, so
Theorem~\ref{thm:komlos} gives
\[
 \disc(A)=\sqrt t\,\disc(A/\sqrt t)<3\sqrt{2\pi t}.
 \qedhere
\]
\end{proof}
\end{samepage}

The argument is existential. No polynomial-time implementation of the
recursive domain construction, its membership queries, and the resulting
sign recovery is established here. We do not claim that the constant
is optimal.

\subsection{Development of the Problem}

Beck and Fiala~\cite{BF81} obtained the classical linear bound
\(2t-1\); see the later account~\cite[Section~1]{BDG19}.
Spencer's \emph{six standard deviations} theorem~\cite{Spe85} gives
\(6\sqrt n\) for \(n\) sets on \(n\) elements, as stated by Alon
and Spencer~\cite[Theorem~13.2.1]{AS16}.
The partial-coloring method fixes a positive fraction of the variables
at each stage. Lovett and Meka~\cite[Sections~1--2]{LM15} give a
constructive version yielding the \(O(\sqrt n)\) order in this regime.
Under a column norm or degree bound, however, progress in the number
of colored variables need not imply comparable progress for each row;
see Bansal, Dadush, and Garg~\cite[Section~1.2]{BDG19}.
Summing the classical partial-coloring estimates gives \(O(\log n)\)
for Koml\'os, as reviewed by Reis and Rothvoss~\cite[Section~1]{RR23}.

Banaszczyk's vector-balancing theorem~\cite{Ban98} introduced a
Gaussian-measure criterion. In the formulation recalled by Dadush, Garg,
Lovett, and Nikolov~\cite[Theorem~1.1]{DGLN19}, vectors of Euclidean norm at most \(1/5\)
can be signed into any convex body of standard Gaussian measure
at least \(1/2\). A cube of fixed half-side length has Gaussian measure
tending to zero with \(m\). The direct cube bound is
\(O(\sqrt{\log m})\); the refined Koml\'os consequence is
\(O(\sqrt{\log n})\), uniformly in \(m\), as explained by
Dadush et al.~\cite[Section~1, p.~3]{DGLN19}.
Bansal, Dadush, and Garg~\cite[Theorem~2]{BDG19} obtained an efficient
randomized algorithm attaining this latter bound for matrices whose
columns have Euclidean norm at most one. The Gram--Schmidt walk of
Bansal, Dadush, Garg, and Lovett samples full colorings with a constant
sub-Gaussian bound on the discrepancy vector relative to the initial
fractional coloring. Together with
the convex-body reduction, this gives an algorithmic form of the general
Gaussian-measure theorem, up to an absolute constant~\cite[Theorems~1.2 and~1.4]{GSW19}.

Bansal and Jiang~\cite[Theorem~1.2, full version]{BJOrig26} establish
the Koml\'os bound \(\widetilde O((\log n)^{1/4})\), where the tilde
suppresses powers of \(\log\log n\).
Their subsequent exposition~\cite[Theorem~1.1]{BJ26} states the explicit bound
\[
 O\bigl((\log n)^{1/4}(\log\log n)^{7/4}\bigr).
\]
The exposition does not optimize the secondary factor.
Its proof controls interactions among rows through affine spectral
independence~\cite[Section~1.1]{BJ26}.

For Beck--Fiala, Altschuler and Tikhomirov~\cite[Corollary~1.2]{AT26}
obtain \(O_\eta(\sqrt t)\) when
\(t\ge C_\eta(\log n)(\log\log n)^{2+\eta}\), for fixed \(\eta>0\)
and sufficiently large \(n\). Their online theorem gives probabilistic
prefix guarantees for input sequences fixed in advance~\cite[Theorem~1.1]{AT26}.
Our Corollary~\ref{cor:beck-fiala} applies to every positive integer
\(t\), without a lower restriction depending on \(n\).

Nikolov~\cite[Theorem~1.1]{Nik13} proved the constant-one bound for
\emph{vector discrepancy}: each scalar sign is replaced by a unit vector,
and each row sum is measured in Euclidean norm.
Reis and Rothvoss~\cite[Theorems~1--2]{RR23} study balancing from
\(\ell_p\) into \(\ell_q\), with their partial-to-full result excluding
\((p,q)=(2,\infty)\).
Guillen and Kobzar~\cite[Theorem~1.1 and Section~1.1]{GK26} connect
a repeated online balancing game to curvature flow in a long-time limit,
a different regime from a single finite signing.

In the opposite direction, Kunisky~\cite[Theorem~1.1]{Kun23} constructs
matrices with column Euclidean norms at most one whose discrepancy
approaches \(1+\sqrt2\).

\subsection{The Variation Approach}

Our geometric framework is an open-set version of Banaszczyk's
convex-body transform~\cite{Ban98}. Its retained fibers expand while
the transformed set stays in \((K-v)\cup(K+v)\), permitting sign
recovery. Szusterman~\cite[Section~1]{Szu25} recalls the original
transform and this recovery argument. The Gaussian-threshold variant
in that paper interacts with Ehrhard symmetrization~\cite[Section~2]{Szu25}.
We use the strict Euclidean fiber-length threshold in
\eqref{eq:transform}.

For a smooth compactly supported density \(\rho\),
\(V_u(\rho)=\int|\partial_u\rho|\) is its first-order
\(L^1\)-translation rate in direction \(u\).
Small variation means that a short translation leaves most mass
overlapping its previous position. The extension to functions of bounded
variation (BV) includes boundary jumps. For \(\kappa\ge0\), we require one density with
\(V_u(\rho)\le\kappa\|u\|_2\) for every \(u\), so that the same
bound applies to vectors in unrelated directions.

Smirnov and Vershynin~\cite[Theorem~1.1 and Section~1.4]{SV26} also
optimize a density, using translation Fisher information to bound
expected discards from a fixed input sequence with independent random signs.
Their online algorithm accepts or discards each incoming term; the
accepted partial sums remain in twice the symmetric convex body
supporting the density.
Their proof~\cite[Section~2]{SV26} already relates translation overlap
to directional \(L^1\)-variation before bounding it by Fisher information.
Here the main step preserves a common directional-variation bound as
the target domain changes, allowing every input vector to receive a sign.
We control the final sum; a corresponding bound for every prefix is not
obtained here. Their product cosine-squared density~\cite[Section~1.6]{SV26}
gives our initial bound by a direct \(L^1\) estimate.

The main analytic step transfers averaged horizontal variation bounds
in a symmetric convex lift to a prescribed section. For a fixed weighted
sum of directional variations, the uniform density on a set of finite
perimeter and positive volume has energy equal to its anisotropic
perimeter divided by its volume.
Minimizing the energy over
admissible probability densities gives an anisotropic Cheeger value.
Eigenvalue convexity, following Wang and Xia~\cite[Theorem~1.2]{WX11},
and a \(p\downarrow1\) limit make that cost convex along the sections.
Appendix~\ref{app:eigenvalue} supplies the eigenvalue argument for the
smooth, strongly convex norms used in the approximation.
Jensen's inequality then bounds this value at the mass-weighted mean
absolute height of the lift.
Convex separation and compactness reconstruct a single density controlling
all directions on that section. An exact rearrangement identity ensures
that height one is available when \(\kappa\|v\|_2\le1/3\).
Thus the transform retains the same numerical bound, although the density
can change at every step.

\noindent\textbf{The role of AI in this proof.}
Odin Automatic AI Research Agent was used to construct the~proof.

\noindent\textbf{Paper organization.}
Section~\ref{sec:signing} reduces full signing to finite-step stability
and estimates the initial cube density. Section~\ref{sec:cheeger}
establishes the prescribed-section principle through Cheeger convexity
and a common-density argument. Section~\ref{sec:stability} combines that
principle with a geometric lift and fiber rearrangement to prove stability.

\Needspace{11\baselineskip} 
\section{From Variation Stability to Full Signings}\label{sec:signing}

All function spaces are over the real numbers. Integrable functions and
their identities are understood up to equality almost everywhere.
Unspecified integrals and all derivatives are taken on the ambient
Euclidean space.
For \(f\in L^1(\R^d)\) and \(u\in\R^d\), define
\begin{equation}\label{eq:variation}
 V_u(f)=
 \sup_{\substack{\phi\in C_c^1(\R^d)\\\|\phi\|_\infty\le1}}
 \int_{\R^d}f\,\partial_u\phi\in[0,\infty].
\end{equation}
When finite, this is the total variation of the distributional
derivative \(D_uf\), namely \(|D_uf|(\R^d)\). Placing an absolute value
around the integral gives the same supremum, since the test class is closed
under \(\phi\mapsto-\phi\). For \(f\in W^{1,1}(\R^d)\), it equals
\(\int|\partial_uf|\). We use \(BV(\R^d)\) for the integrable functions
whose full distributional gradient is a finite vector measure, and
write \(|Df|(\R^d)\) for its total variation.
For probability densities \(f,g\), the total variation \emph{distance}
between their laws is a different quantity:
\begin{equation}\label{eq:tv-distance}
 d_{\mathrm{TV}}(f\,dx,g\,dx)=\tfrac12\|f-g\|_1.
\end{equation} 

For a nonempty bounded open convex set \(K\subset\R^d\), let
\(\Prob(K)\) be the nonnegative functions \(\rho\in BV(\R^d)\) with
\[
 \int\rho=1,\qquad
 \rho=0\quad\text{almost everywhere on }\R^d\setminus K.
\]
Throughout, ``supported in \(K\)'' has this almost-everywhere meaning.
The topological support may meet \(\partial K\); compact support inside
\(K\) will be required only for specific smooth approximations.
For \(\kappa\ge0\), the invariant is the existence of
\(\rho\in\Prob(K)\) satisfying
\begin{equation}\label{eq:cap}
 V_u(\rho)\le\kappa\|u\|_2
 \qquad\text{for every }u\in\R^d.
\end{equation}
The global derivative is important even for elementary densities.
For \(C>0\), the density \(\rho=(2C)^{-1}\mathbf1_{(-C,C)}\) is constant inside
its support, but its two endpoint jumps give \(V_1(\rho)=1/C\).
For \(|h|\le2C\), its \(L^1\) translation change is exactly \(|h|/C\).
Computing variation only inside the open interval would miss it.

For a vector \(v\in\R^d\), define
\begin{equation}\label{eq:transform}
\Tv_vK=((K-v)\cap(K+v))+\{tv:-2<t<2\}.
\end{equation}
This is an open-set version of Banaszczyk's convex-body transform
\cite{Ban98}, whose closed-set definition and use in sign recovery are
recalled by Szusterman~\cite[Section~1, pp.~1--2]{Szu25}.
When \(v\ne0\), fibers parallel to \(v\) of length at most
\(2\|v\|_2\) disappear; longer fibers grow by \(\|v\|_2\) at each end.
The closed-set convention retains fibers at the threshold, whereas
\eqref{eq:transform} uses a strict inequality. For example,
\(\Tv_1(-1,1)=\varnothing\), while the closed interval \([-1,1]\)
is transformed into \([-2,2]\). Thus taking the interior of the
closed-set transform need not give~\eqref{eq:transform}.
The variation estimate below must control the loss of short fibers.

\begin{samepage}
\begin{lemma}[Translate containment]\label{lem:containment}
For a bounded open convex set \(K\), the set \(\Tv_vK\), when nonempty,
is bounded, open, and convex, and
\begin{equation}\label{eq:containment}
 \Tv_vK\subseteq(K-v)\cup(K+v).
\end{equation}
If \(K=-K\), then \(\Tv_vK=-\Tv_vK\).
\end{lemma}
\begin{proof}
The geometric properties follow from~\eqref{eq:transform}.
Write \(y=z+tv\), with \(z-v,z+v\in K\) and \(|t|<2\).
If \(t\ge0\), convexity gives \(y-v=z+(t-1)v\in K\).
If \(t\le0\), it gives \(y+v=z+(t+1)v\in K\).
This proves~\eqref{eq:containment}.
\end{proof}
\end{samepage}

\begin{proposition}[Finite-step stability]\label{prop:stability}
Let \(K\subset\R^d\) be nonempty, bounded, open, and convex.
Let \(\kappa\ge0\), \(v\in\R^d\), and suppose
\(\rho\in\Prob(K)\) satisfies~\eqref{eq:cap}.
If
\begin{equation}\label{eq:small-step}
 \kappa\|v\|_2\le\frac13,
\end{equation}
then \(\Tv_vK\) is nonempty and admits
\(\widetilde\rho\in\Prob(\Tv_vK)\) with
\(V_u(\widetilde\rho)\le\kappa\|u\|_2\) for every \(u\in\R^d\).
\end{proposition}

The proof in Section~\ref{sec:stability} only needs \(V_v(\rho)\le1/3\)
as its step condition; \eqref{eq:small-step} ensures this
under~\eqref{eq:cap}. The stated form gives the following criterion
for iterated signing.

\begin{samepage}
\begin{corollary}[Symmetric targets]\label{cor:target}
Let \(K\subset\R^d\) be nonempty, bounded, open, and convex, with
\(K=-K\). Suppose that, for some \(\kappa\ge0\), there is
\(\rho\in\Prob(K)\) satisfying~\eqref{eq:cap}.
For every integer \(n\ge1\) and family \(v_1,\ldots,v_n\in\R^d\) with
\(\kappa\max_j\|v_j\|_2\le1/3\), there are signs
\(\varepsilon_1,\ldots,\varepsilon_n\in\{-1,1\}\) such that
\[
 \sum_{j=1}^n\varepsilon_jv_j\in K.
\]
\end{corollary}
\begin{proof}
Set \(K_0=K\) and \(K_j=\Tv_{v_j}K_{j-1}\) for \(1\le j\le n\).
Proposition~\ref{prop:stability} applies at every stage with the same
\(\kappa\). Each \(K_j\) is nonempty, bounded, open, and convex.
Lemma~\ref{lem:containment} preserves symmetry about the origin.
If \(z\in K_n\), then \(-z\in K_n\), so convexity gives \(0\in K_n\).
This is the only use of symmetry in the signing reduction; the density
chosen at each stage need not be symmetric.

Starting from \(z_n=0\), choose successively
\(\varepsilon_j\in\{-1,1\}\), for \(j=n,\ldots,1\), so that
\[
 z_{j-1}=z_j+\varepsilon_jv_j\in K_{j-1}.
\]
Such a choice exists by~\eqref{eq:containment}. Consequently
\[
 \sum_{j=1}^n\varepsilon_jv_j=z_0\in K.
 \qedhere
\]
\end{proof}
\end{samepage}

The initial bound comes from the product cosine-squared density used
by Smirnov and Vershynin~\cite[Section~1.6, proof of Corollary~1.3]{SV26}.
It is the square of the \(L^2\)-normalized first Dirichlet eigenfunction
of the cube.
We estimate its directional \(L^1\)-variation.

\begin{lemma}[A cube density]\label{lem:cube}
For \(C>0\), the probability density
\begin{equation}\label{eq:cube-density}
 \rho_C(x)=C^{-d}\prod_{i=1}^d
       \cos^2\!\left(\frac{\pi x_i}{2C}\right)
       \mathbf1_{(-C,C)^d}(x)
\end{equation}
belongs to \(\Prob((-C,C)^d)\) and satisfies
\begin{equation}\label{eq:cube-cap}
 V_u(\rho_C)\le\frac{\sqrt{2\pi}}C\,\|u\|_2
 \qquad(u\in\R^d).
\end{equation}
\end{lemma}
\begin{samepage}
\begin{proof}
Each normalized one-dimensional factor has integral one and vanishes at its
endpoints. The zero extension is in \(W^{1,1}(\R^d)\).
If \(X\) has density \(\rho_C\), the variables
\(T_i=\tan(\pi X_i/(2C))\) are independent with density
\(2/(\pi(1+t^2)^2)\). Differentiation gives
\[
 V_u(\rho_C)=\frac\pi C\,
       \mathbb E\left|\sum_i u_iT_i\right|.
\]
The independent \(T_i\) admit the joint representation in distribution
\(T_i=G_i/\sqrt{S_i}\), where the \(G_i\)
are standard Gaussian variables and the \(S_i\) have the
\(\chi^2_3\) distribution, all independent. Indeed,
\[
 \int_0^\infty
 \sqrt{\frac{s}{2\pi}}e^{-st^2/2}
 \frac{s^{1/2}e^{-s/2}}{\sqrt{2\pi}}\dd s
 =\frac{2}{\pi(1+t^2)^2},
\]
and the same gamma integral gives \(\mathbb E S_i^{-1}=1\).
Conditioning on the \(S_i\) and applying Jensen yields
\[
 \begin{aligned}
 \mathbb E\left|\sum_i u_iT_i\right|
 &=\sqrt{\frac2\pi}\,
    \mathbb E\sqrt{\sum_i\frac{u_i^2}{S_i}}\\
 &\le\sqrt{\frac2\pi}\,
    \sqrt{\sum_i u_i^2\mathbb E S_i^{-1}}
 =\sqrt{\frac2\pi}\,\|u\|_2.
 \end{aligned}
\]
This proves~\eqref{eq:cube-cap}.
The \(T_i\) are centered, with
\(\mathbb E T_i^2=\mathbb E S_i^{-1}=1\), so Cauchy--Schwarz
alone would give \(V_u(\rho_C)\le(\pi/C)\|u\|_2\).
Conditioning on the \(S_i\) improves the constant
to \(\sqrt{2\pi}/C\).
\end{proof}
\end{samepage}

\begin{proof}[Proof of Theorem~\ref{thm:komlos}]
Set \(C=3\sqrt{2\pi}\), \(K=(-C,C)^m\), and \(\kappa=1/3\).
Lemma~\ref{lem:cube} supplies a density satisfying~\eqref{eq:cap}.
Since \(\|v_j\|_2\le1\), Corollary~\ref{cor:target} gives a signing
whose sum lies in \(K\). Membership in this open cube gives the
strict inequality in~\eqref{eq:main}.
\end{proof}

\begin{remark}[Sharpness for the product density]\label{rem:cube-sharp}
The coefficient \(\sqrt{2\pi}\) in~\eqref{eq:cube-cap} is optimal
uniformly over dimension for this density family. To see this, fix
\(C>0\), write \(\rho_C^{(d)}\) for~\eqref{eq:cube-density} in
\(\R^d\), and take \(u_d=d^{-1/2}(1,\ldots,1)\).
The conditional Gaussian calculation gives
\[
 C V_{u_d}(\rho_C^{(d)})
 =\sqrt{2\pi}\,\mathbb E\sqrt{\frac1d\sum_{i=1}^d S_i^{-1}},
\]
where the \(S_i\) are independent with law \(\chi^2_3\).
Since \(\mathbb E S_i^{-1}=1\), the strong law makes the square root
tend almost surely to one. Fatou bounds the lower limit of these
expectations from below by one, while Jensen bounds each expectation
from above by one. Thus
\(C V_{u_d}(\rho_C^{(d)})\to\sqrt{2\pi}\).
This concerns the fixed product family; it does not establish optimality
among all cube-supported densities or for the Koml\'os constant.
\end{remark}

\section{A Cheeger Principle for Convex Sections}\label{sec:cheeger}

The lift will control horizontal variations averaged over its sections.
We seek the same bounds on one prescribed section. We first treat finite
weighted sums of directional variations, then use convex separation
to recover a common density.

We record the elementary variation facts used in this passage.

\begin{lemma}[Translation and compactness]\label{lem:bv}
For every \(f\in L^1(\R^d)\), \(u\in\R^d\), and \(h\ne0\),
\begin{equation}\label{eq:translation}
 \|f(\cdot+hu)-f\|_1\le |h|V_u(f),\qquad
 V_u(f)=\lim_{h\to0}\frac{\|f(\cdot+hu)-f\|_1}{|h|},
\end{equation}
where the equality allows the value \(+\infty\).
The functional \(V_u\) is convex and lower semicontinuous in \(L^1\).
Probability densities supported in a fixed bounded set and with
uniformly bounded coordinate variations have an \(L^1\)-convergent
subsequence.
\end{lemma}
\begin{proof}
For finite \(V_u(f)\), mollify \(f\). The directional derivative of the
mollification has \(L^1\) norm at most \(V_u(f)\). Integrating along
segments and passing to the \(L^1\) limit gives the first inequality.
Conversely, for each test function in~\eqref{eq:variation},
integration against difference quotients and passage to the limit give
\[
 \left|\int f\,\partial_u\phi\right|
 \le\|\phi\|_\infty
    \liminf_{h\to0}\frac{\|f(\cdot+hu)-f\|_1}{|h|}.
\]
Taking the supremum proves the equality, including the infinite case.
Convexity and lower semicontinuity follow directly from the same
test-function definition.

Coordinate bounds give a uniform bound on \(|Df|(\R^d)\), since
\(|Df|(\R^d)\le\sum_iV_{e_i}(f)\).
Successive translations in the coordinate directions give
\[
 \|f(\cdot+h)-f\|_1
 \le\sum_{i=1}^d |h_i|V_{e_i}(f),\qquad h\in\R^d,
\]
so the translation modulus is uniform over the family.
Approximating each density by its averages on cubes of a fixed small
mesh therefore gives uniformly small \(L^1\) errors. These averages lie
in a bounded subset of a finite-dimensional space, because the supports
are bounded and the masses are one. A diagonal selection over
successively finer meshes gives an \(L^1\)-Cauchy subsequence,
proving the compactness assertion.
\end{proof}

For a continuous, even, convex, positively one-homogeneous function
\(H:\R^d\to[0,\infty)\), \(\rho\in BV(\R^d)\), and a nonempty
bounded open convex set \(K\subset\R^d\), define
\begin{equation}\label{eq:cheeger}
 E_H(\rho)=\int_{\R^d}
 H\!\left(\frac{dD\rho}{d|D\rho|}\right)\dd|D\rho|,
 \qquad
 h_H(K)=\inf_{\rho\in\Prob(K)}E_H(\rho).
\end{equation}
Here \(dD\rho/d|D\rho|\) is the Radon--Nikodym density of the
vector measure \(D\rho\) with respect to its total variation.
Writing it as \(\sigma_\rho\), we have
\[
 D_u\rho=(u\cdot\sigma_\rho)|D\rho|,
 \qquad
 V_u(\rho)=\int|u\cdot\sigma_\rho|\dd|D\rho|.
\]
The definition immediately gives domain monotonicity:
\(K\subset L\) implies \(h_H(L)\le h_H(K)\).
For a bounded set \(E\) with smooth boundary and positive volume,
the uniform density has energy
\[
 E_H\!\left(\frac{\mathbf1_E}{|E|}\right)
 =\frac1{|E|}\int_{\partial E}H(\nu_E)\dd\mathcal H^{d-1},
\]
where \(\nu_E\) is the unit normal and \(\mathcal H^{d-1}\) is surface
measure. Thus \(H\) assigns a cost to each boundary orientation, and
the normalization compares that cost with the volume carrying the mass.
More generally, write
\(P_H(E)=\int_{\partial^*E}H(\nu_E)\,d\mathcal H^{d-1}\)
for a set of finite perimeter, where \(\partial^*E\) is its
reduced~boundary. The anisotropic BV coarea formula, as used by Kawohl and
Novaga~\cite[proof of Theorem~4.1]{KN08}, and layer cake give, for
\(\rho\in\Prob(K)\),
\[
 E_H(\rho)=\int_0^\infty P_H(\{\rho>s\})\dd s,
 \qquad
 1=\int_0^\infty |\{\rho>s\}|\dd s.
\]
Almost every superlevel set at a positive level has finite perimeter.
The energy is at least the infimum of the perimeter-to-volume ratios
of those with positive volume. Uniform densities
give the reverse inequality, so
\begin{equation}\label{eq:cheeger-perimeter}
 h_H(K)=
 \inf_{\substack{|E\setminus K|=0,\ |E|>0\\
                  E\ \text{of finite perimeter}}}
          \frac{P_H(E)}{|E|}.
\end{equation}
For a norm \(H\), this is the anisotropic Cheeger constant of
Kawohl and Novaga~\cite[equations~(5),~(10), and proof of Theorem~4.1]{KN08}.
In their notation, our gradient and normal integrand \(H\) corresponds
to the dual norm \(\phi^*\), not to the norm \(\phi\) defining distance.
Perimeters are computed in \(\R^d\), including any reduced boundary
lying on \(\partial K\).

We use integrands that are finite convex combinations of absolute
linear forms:
\begin{equation}\label{eq:mixture}
 H(\xi)=\sum_{\ell=1}^N\alpha_\ell|u_\ell\cdot\xi|,
 \qquad
 \|u_\ell\|_2=1,\quad \alpha_\ell\ge0,\quad
 \sum_\ell\alpha_\ell=1.
\end{equation}
The corresponding energy is
\(E_H(\rho)=\sum_\ell\alpha_\ell V_{u_\ell}(\rho)\).
Such an \(H\) can vanish in nonzero directions and need not be a norm.
The connection with first Dirichlet eigenvalues of anisotropic
\(p\)-Laplacians is treated by Kawohl and Novaga~\cite[Theorem~4.1]{KN08}.
We give the limit argument for the global BV class in~\eqref{eq:cheeger},
then pass to possibly degenerate combinations~\eqref{eq:mixture}.

\begin{lemma}[Minkowski convexity]\label{lem:minkowski}
For \(H\) as in~\eqref{eq:mixture}, nonempty bounded open convex
sets \(K_0,K_1\subset\R^d\), and \(0\le t\le1\),
\begin{equation}\label{eq:cheeger-convexity}
 h_H((1-t)K_0+tK_1)
 \le(1-t)h_H(K_0)+t h_H(K_1).
\end{equation}
\end{lemma}
\begin{proof}
First let \(H\) be a norm, \(C^\infty\) away from the origin,
with \(D^2(H^2)\) positive definite there. Thus \(H\) is even,
positive away from zero, and homogeneous of degree one.
For \(1<p<\infty\), put
\begin{equation}\label{eq:eigenvalue}
 \lambda_{p,H}(K)=
 \inf_{0\ne f\in W_0^{1,p}(K)}
       \frac{\int_K H(\nabla f)^p}{\int_K|f|^p}.
\end{equation}
Norm equivalence and the Dirichlet Poincar\'e inequality make this
quantity positive and finite for every nonempty bounded open \(K\).
With \(K_t=(1-t)K_0+tK_1\), we need
\begin{equation}\label{eq:eigenvalue-convexity}
 \lambda_{p,H}(K_t)
 \le(1-t)\lambda_{p,H}(K_0)+t\lambda_{p,H}(K_1).
\end{equation}
This follows from Wang and Xia's inverse-power
inequality~\cite[Theorem~1.2]{WX11}, with their integrand \(F=H\)
and the same Rayleigh normalization. For the stated strongly convex
class, Proposition~\ref{prop:eigenvalue-convexity} gives a proof using
the eigenfunction results of Mosconi, Riey, and
Squassina~\cite[Theorem~1.1 and Proposition~4.5]{MRS24}.
It handles critical points by a weak cutoff argument and removes smooth
boundary assumptions by inner exhaustion. No differentiability of
\(H^2\) at zero is required.

We claim that
\begin{equation}\label{eq:p-limit}
 \lim_{p\downarrow1}\lambda_{p,H}(K)=h_H(K).
\end{equation}
First, smooth compactly supported probability densities suffice for
the infimum in~\eqref{eq:cheeger}. Choose \(x_0\in K\) and
\(\delta_0>0\) with \(B(x_0,\delta_0)\subset K\). For \(0<r<1\),
convexity gives
\[
 K_r:=x_0+r(\overline K-x_0)\Subset K,
 \qquad \operatorname{dist}(K_r,K^c)\ge(1-r)\delta_0.
\]
For a feasible density \(\rho\), contraction gives a probability density
supported in \(K_r\), with energy \(r^{-1}E_H(\rho)\), namely
\[
 \rho_r(x)=r^{-d}\rho\!\left(x_0+\frac{x-x_0}{r}\right).
\]
Convolve its global zero extension
with a nonnegative smooth kernel of mass one and radius less than
\((1-r)\delta_0\). The result belongs to \(C_c^\infty(K)\), has
mass one, and has energy at most \(E_H(\rho_r)\), by convexity and
one-homogeneity of \(H\). Letting \(r\uparrow1\) proves the assertion.
This approximation includes the derivative measure on the support boundary.

For \(0\ne f\in W_0^{1,p}(K)\), its zero extension has
\(|f|^p\in W^{1,1}(\R^d)\). Testing with
\(\rho=|f|^p/\int|f|^p\), the chain rule and H\"older give
\[
 h_H(K)
 \le p\left(\frac{\int H(\nabla f)^p}{\int|f|^p}\right)^{1/p}.
\]
Consequently,
\[
 \lambda_{p,H}(K)\ge\bigl(h_H(K)/p\bigr)^p,
 \qquad
 \liminf_{p\downarrow1}\lambda_{p,H}(K)\ge h_H(K).
\]
For the upper limit, fix \(\delta>0\) and then choose a nonnegative
\(f_\delta\in C_c^\infty(K)\) with mass one and
\(E_H(f_\delta)\le h_H(K)+\delta\). This same function is a
Rayleigh test for every \(p>1\), and dominated convergence gives
\[
 \begin{aligned}
 \limsup_{p\downarrow1}\lambda_{p,H}(K)
 &\le\lim_{p\downarrow1}
       \frac{\int_K H(\nabla f_\delta)^p}{\int_K f_\delta^p}\\
 &= E_H(f_\delta)\le h_H(K)+\delta.
 \end{aligned}
\]
Let \(\delta\downarrow0\). This proves~\eqref{eq:p-limit}; applying it
to \(K_0,K_1,K_t\) in the eigenvalue inequality gives~\eqref{eq:cheeger-convexity}.

For the mixture~\eqref{eq:mixture} and \(\epsilon>0\), set
\[
 H_\epsilon(\xi)=
 \sum_\ell\alpha_\ell
 \sqrt{(u_\ell\cdot\xi)^2+\epsilon^2\|\xi\|_2^2}.
\]
Each ellipsoidal norm in this sum is smooth away from zero and has
nonnegative Hessian with kernel exactly the radial direction.
Since the nonnegative weights sum to one, \(D^2H_\epsilon(\xi)\)
is positive semidefinite with kernel \(\operatorname{span}\{\xi\}\)
for every \(\xi\ne0\). In
\[
 D^2(H_\epsilon^2)
 =2\nabla H_\epsilon\otimes\nabla H_\epsilon
   +2H_\epsilon D^2H_\epsilon,
\]
the second term is positive on every nonradial vector. On radial vectors
the first term is positive, by
\(\nabla H_\epsilon(\xi)\cdot\xi=H_\epsilon(\xi)>0\).
Thus \(D^2(H_\epsilon^2)\) is positive definite away from zero.
These are precisely the smooth, strongly convex integrands covered by
Proposition~\ref{prop:eigenvalue-convexity} and the preceding limit argument.
Moreover,
\[
 H\le H_\epsilon\le H+\epsilon\|\cdot\|_2.
\]
For \(\delta>0\), fix \(\rho_\delta\in\Prob(K)\) with
\(E_H(\rho_\delta)\le h_H(K)+\delta\). Then
\[
 h_H(K)\le h_{H_\epsilon}(K)
 \le h_H(K)+\delta+\epsilon|D\rho_\delta|(\R^d).
\]
First let \(\epsilon\downarrow0\) with this density fixed, and then
let \(\delta\downarrow0\). It follows that
\(h_{H_\epsilon}(K)\to h_H(K)\). Passing to this limit at
\(K_0,K_1,K_t\) proves~\eqref{eq:cheeger-convexity}, even if the
directions do not span \(\R^d\). No uniform BV bound for minimizers
of the degenerating energies is required.
\end{proof}

The next lemma recovers one density from bounds on all finite weighted
infima. Attainment of the individual weighted infima is not required:
separation followed by compactness supplies exact simultaneous
directional bounds.

\begin{lemma}[A common density]\label{lem:common}
Let \(K\subset\R^d\) be nonempty, bounded, open, and convex.
Let \(\kappa\ge0\).
If \(h_H(K)\le\kappa\) for every mixture~\eqref{eq:mixture}, then
there is \(\rho\in\Prob(K)\) satisfying~\eqref{eq:cap}.
\end{lemma}
\begin{proof}
Fix finitely many unit directions \(u_1,\ldots,u_N\) including the
coordinate basis. Including this basis will turn bounds on the listed
directions into the BV bound
\[
 |D\rho|(\R^d)\le\sum_{i=1}^dV_{e_i}(\rho)
 \le d\max_{1\le\ell\le N}V_{u_\ell}(\rho).
\]
Its dimension dependence is used only to obtain compactness in the
fixed space \(\R^d\), not in the eventual directional bound.
The set
\[
 S=\{z\in\R^N:\exists\rho\in\Prob(K),\
       z_\ell\ge V_{u_\ell}(\rho)\ (1\le\ell\le N)\}
\]
is nonempty, convex, and upward closed in every coordinate.
If \(\inf_{\rho\in\Prob(K)}\max_\ell V_{u_\ell}(\rho)>\kappa\),
choose \(a\) strictly between these values. The convex sets \(S\)
and \(Q=(-\infty,a)^N\) are disjoint, and \(Q\) is open.
The separation theorem for a convex set and a disjoint open convex set
provides a nonzero linear functional \(\beta\) with
\[
 \inf_{z\in S}\beta\cdot z\ge\sup_{q\in Q}\beta\cdot q.
\]
Since \(S+\tau e_\ell\subset S\) for all \(\tau\ge0\), each
\(\beta_\ell\) must be nonnegative. Divide by
\(\sum_\ell\beta_\ell>0\) to obtain weights \(\alpha_\ell\).
The supremum over \(Q\) is now \(a\), so
\[
 \inf_{\rho\in\Prob(K)}
       \sum_\ell\alpha_\ell V_{u_\ell}(\rho)\ge a>\kappa,
\]
contrary to the hypothesis. Thus the finite-direction infimum is at
most \(\kappa\).

A minimizing sequence for the finite maximum has bounded coordinate
variations. Lemma~\ref{lem:bv} gives a subsequence converging in
\(L^1\) to \(\rho\). Nonnegativity and mass one pass to the limit,
and the support condition follows directly from
\[
 \int_{K^c}|\rho|
 \le\|\rho-\rho_j\|_1\longrightarrow0.
\]
The coordinate bounds also give \(\rho\in BV(\R^d)\).
Lower semicontinuity of the finitely many variations shows that the
minimum is attained, with every listed variation at most \(\kappa\).

Now take increasing finite lists containing the coordinate basis and
a countable dense set of unit directions. The corresponding minimizing
densities satisfy \(|D\rho|(\R^d)\le d\kappa\).
Another application of compactness gives a single limit density
in \(\Prob(K)\). Each fixed direction from the countable list occurs
in all sufficiently late lists, so lower semicontinuity preserves its
bound along this one subsequence. Finally,
\[
 V_u(\rho)\le V_w(\rho)+\|u-w\|_2\,|D\rho|(\R^d)
\]
extends the bound to all unit \(u\), and homogeneity gives~\eqref{eq:cap}.
\end{proof}

We now pass from averaged horizontal variation bounds to a prescribed section.

\begin{lemma}[Prescribed sections]\label{lem:section}
Let \(B\subset\R^d\times\R\) be nonempty, bounded, open, and convex,
and invariant under \((y,s)\mapsto(y,-s)\). Put
\(D_t=\{y:(y,t)\in B\}\), \(t\ge0\).
Let \(\kappa\ge0\). Suppose \(R\in L^1(\R^{d+1})\) is a probability density, vanishes
almost everywhere outside \(B\), and satisfies
\[
 V_{(u,0)}(R)\le\kappa
 \qquad(\|u\|_2=1).
\]
If \(a\ge0\) and \(\int|s|R\ge a\), then
\(D_a\) is nonempty and admits \(\rho\in\Prob(D_a)\)
satisfying~\eqref{eq:cap}.
\end{lemma}
\begin{proof}
The projection of \(B\) on the last coordinate is an interval
\((-b,b)\), with \(0<b<\infty\). Thus precisely the sections with
\(0\le t<b\) are nonempty.
Reflection and convexity give \(D_t\subseteq D_s\) for
\(0\le s\le t<b\), and
\[
 (1-\theta)D_s+\theta D_t
 \subseteq D_{(1-\theta)s+\theta t}
 \qquad(0\le s,t<b,\ 0\le\theta\le1).
\]
For a fixed mixture \(H\), domain monotonicity gives that
\(g(t):=h_H(D_t)\) is nondecreasing. The same monotonicity and
Lemma~\ref{lem:minkowski} give
\[
 \begin{aligned}
 h_H(D_{(1-\theta)s+\theta t})
 &\le h_H((1-\theta)D_s+\theta D_t)\\
 &\le(1-\theta)h_H(D_s)+\theta h_H(D_t).
 \end{aligned}
\]
Hence \(g\) is also convex on \([0,b)\). It is finite there, since
each nonempty section admits a smooth compactly supported probability
density.

For almost every \(s\), let \(R_s=R(\cdot,s)\) and
\(q(s)=\int_{\R^d}R(y,s)\dd y\). For each fixed direction \(u\),
the difference-quotient formula~\eqref{eq:translation}, Fubini, and
Fatou give
\begin{equation}\label{eq:horizontal-slicing}
 \begin{aligned}
 \int_\R V_u(R_s)\dd s
 &=\int_\R\lim_{h\to0}
        \frac{\|R_s(\cdot+hu)-R_s\|_1}{|h|}\dd s\\
 &\le\liminf_{h\to0}
        \frac{\|R(\cdot+(hu,0))-R\|_1}{|h|}
 =V_{(u,0)}(R).
 \end{aligned}
\end{equation}
Intersecting the full-measure sets for the finitely many coordinate
directions gives \(R_s\in BV(\R^d)\) for almost every \(s\).
Fubini, the support condition, and the reflection symmetry of \(B\)
then give \(R_s/q(s)\in\Prob(D_{|s|})\) for almost every \(s\)
with \(q(s)>0\). If \(q(s)=0\), nonnegativity gives \(R_s=0\)
almost everywhere and its energy contributes zero.

The measure \(q(s)\dd s\) is a probability measure concentrated on
\((-b,b)\). Its mean absolute height \(M=\int|s|q(s)\dd s\)
satisfies \(M\ge a\), and
\[
 b-M=\int_{-b}^b(b-|s|)q(s)\dd s>0.
\]
Thus \(D_M\) and \(D_a\) are nonempty, including when \(M=a\).
Normalization of a slice divides its energy by \(q(s)\); its
contribution is therefore weighted by that slice's mass. For the fixed
mixture \(H\), this yields
\[
 \int_{-b}^b q(s)h_H(D_{|s|})\dd s
 \le\sum_\ell\alpha_\ell\int_\R V_{u_\ell}(R_s)\dd s
 \le\kappa.
\]
This estimate makes \(g(|s|)\) integrable with respect to
\(q(s)\dd s\). Applying Jensen under the distribution of
\(|s|\) induced by \(q(s)\dd s\), and then using its monotonicity, gives
\[
 \kappa\ge\int_{-b}^b q(s)h_H(D_{|s|})\dd s
 \ge h_H(D_M)\ge h_H(D_a).
\]
This holds for every mixture \(H\).
Lemma~\ref{lem:common} now supplies the required density on \(D_a\).
\end{proof}

\begin{remark}
Only the region \(B\) must be reflection symmetric. Neither \(R\)
nor its marginal \(q\) is assumed symmetric, and no vertical derivative
of \(R\) is used. The density obtained on \(D_a\) is a new common
density, not necessarily one of the normalized input slices.
\end{remark}

\section{Lifting and Finite-Step Stability}\label{sec:stability}

We now prove Proposition~\ref{prop:stability}.
The cutoff three makes the height-one section equal to the sign transform:
an open interval in \((-3,3)\) has length greater than two exactly when it
contains \([c-1,c+1]\) for some \(|c|<2\).

\begin{lemma}[The geometric lift]\label{lem:lift}
For nonempty bounded open convex \(K\subset\R^d\) and \(v\in\R^d\),
define
\[
 B=\{(y,s):|s|<3,\ y+sv\in K\}.
\]
Let \(\ell(y)\) be the length of its vertical fiber, zero for an empty
fiber. The Steiner symmetral
\[
 B^\star=\{(y,s):|s|<\ell(y)/2\}
\]
is nonempty, bounded, open, convex, and symmetric in \(s\). Its
height-one section is
\begin{equation}\label{eq:section-transform}
 \{y:\ell(y)>2\}=\Tv_vK.
\end{equation}
\end{lemma}
\begin{proof}
The lift \(B\) is nonempty, bounded, open, and convex.
Its nonempty vertical fibers \(J_y\) are intervals, and
\[
 (1-\theta)J_{y_0}+\theta J_{y_1}
 \subseteq J_{(1-\theta)y_0+\theta y_1}.
\]
Their lengths form a positive concave function on the open convex
projection of \(B\), hence are continuous there. This proves the stated
properties of \(B^\star\).

If \(J_y=(\alpha,\beta)\) has length greater than two, choose
\(c\in(\alpha+1,\beta-1)\subseteq(-2,2)\).
Here \(c\) is the center of a length-two segment in the vertical
fiber, and \(z=y+cv\) is its corresponding midpoint in \(K\).
Thus \(z\pm v\in K\), so \(y\in\Tv_vK\).
Conversely, if \(y=z+tv\), \(|t|<2\), and \(z\pm v\in K\),
then \(-t-1,-t+1\in J_y\). Openness gives \(\ell(y)>2\).
This also covers \(v=0\).
\end{proof}

The next lemma relates the geometric symmetral \(B^\star\) to fiberwise
density rearrangement by showing how support is preserved.

\begin{lemma}[Fiber rearrangement]\label{lem:rearrangement}
For nonnegative \(R\in L^1(\R^d\times\R)\), put
\begin{equation}\label{eq:fiber-rearrangement}
 \mu(y,t)=\int_\R\mathbf1_{\{R(y,s)>t\}}\dd s,
 \qquad
 R^\star(y,s)=\int_0^\infty
                  \mathbf1_{\{2|s|<\mu(y,t)\}}\dd t,
\end{equation}
where \(t>0\), and set \(R^\star=0\) on the exceptional
nonintegrable fibers. The formula is understood almost everywhere,
with value zero chosen wherever it is infinite. Then \(R^\star\)
is jointly measurable, nonnegative, and has the same mass as \(R\).
If \(R=0\) almost everywhere outside the lift \(B\) in
Lemma~\ref{lem:lift}, then \(R^\star=0\) almost everywhere outside
\(B^\star\). For every \(u\in\R^d\) and \(h\in\R\),
\begin{equation}\label{eq:horizontal-contraction}
 \|R^\star(\cdot+(hu,0))-R^\star\|_1
 \le\|R(\cdot+(hu,0))-R\|_1.
\end{equation}
Consequently \(V_{(u,0)}(R^\star)\le V_{(u,0)}(R)\).
\end{lemma}
\begin{proof}
Tonelli makes \(\mu\) and the function defined by the integral formula
measurable, and gives
\[
 \int_{\R^{d+1}}R^\star
 =\int_{\R^d}\int_0^\infty\mu(y,t)\dd t\dd y
 =\int_{\R^{d+1}}R.
\]
The integral formula is therefore finite almost everywhere, so the
chosen finite-valued representative has the same mass. For almost every
\(y\), the function \(\mu(y,\cdot)\) is decreasing and right-continuous.
Up to null sets, the superlevel sets of \(R^\star(y,\cdot)\) at positive
levels are centered intervals of the same lengths as those of
\(R(y,\cdot)\).
Under the support hypothesis, for almost every \(y\) and every \(t>0\),
the original superlevel set is contained, up to a null set, in \(J_y\).
Hence \(\mu(y,t)\le\ell(y)\), and~\eqref{eq:fiber-rearrangement}
vanishes when \(2|s|\ge\ell(y)\).

For nonnegative \(f,g\in L^1(\R)\), layer cake and the maximal
intersection of centered intervals give
\[
 \int\min\{f^\star,g^\star\}
 =\int_0^\infty\min\{|\{f>t\}|,|\{g>t\}|\}\dd t
 \ge\int\min\{f,g\}.
\]
Together with mass preservation and
\(\|f-g\|_1=\int f+\int g-2\int\min\{f,g\}\), this gives
\(\|f^\star-g^\star\|_1\le\|f-g\|_1\).
Apply it to the fibers at \(y+hu\) and \(y\), and integrate in \(y\).
This proves~\eqref{eq:horizontal-contraction}; the difference-quotient
limit in Lemma~\ref{lem:bv} proves the horizontal variation bound.
\end{proof}

The next identity computes the height retained by this rearrangement.

\begin{lemma}[Rearranged height]\label{lem:height}
Let \(\rho\in BV(\R^d)\) be a probability density and \(v\in\R^d\).
Set
\begin{equation}\label{eq:lifted-density}
 R(y,s)=\frac16\rho(y+sv)\mathbf1_{\{|s|<3\}},
\end{equation}
and rearrange \(R(y,\cdot)\) symmetrically and decreasingly for almost every \(y\),
obtaining \(R^\star\). Then
\begin{equation}\label{eq:height-identity}
 \begin{aligned}
 \int|s|R^\star
 &=\frac32-\frac1{24}\int_0^6(6-h)
       \|\rho-\rho(\cdot+hv)\|_1\dd h\\
 &\ge\frac32-\frac32 V_v(\rho).
 \end{aligned}
\end{equation}
\end{lemma}
\begin{proof}
For a level \(t>0\), let
\(\mu(y,t)=|\{s\in(-3,3):\rho(y+sv)>6t\}|\), as
in~\eqref{eq:fiber-rearrangement}.
The corresponding superlevel interval of \(R^\star(y,\cdot)\)
has first absolute moment \(\mu(y,t)^2/4\).
Layer cake and Tonelli give
\[
 \begin{aligned}
 \int|s|R^\star(y,s)\dd y\dd s
 &=\frac14\int_{\R^d}\int_0^\infty\mu(y,t)^2\dd t\dd y\\
 &=\frac1{24}\int_{-3}^3\int_{-3}^3\int_{\R^d}
        \min\{\rho(y+sv),\rho(y+rv)\}\dd y\dd r\dd s.
 \end{aligned}
\]
The change of level variable \(\tau=6t\) contributes the factor \(1/6\).
For \(h\in\R\), the overlap
\begin{equation}\label{eq:overlap}
 O(h)=\int\min\{\rho(y),\rho(y+hv)\}\dd y
     =1-\frac12\|\rho-\rho(\cdot+hv)\|_1
\end{equation}
is even in \(h\). Integration over the square \((-3,3)^2\)
therefore turns the preceding expression into
\(\frac1{12}\int_0^6(6-h)O(h)\dd h\), proving the identity.
Only the probability-density hypothesis was needed for the exact identity.
By~\eqref{eq:translation}, the density difference is at most
\(hV_v(\rho)\). Since \(\int_0^6h(6-h)\dd h=36\),
this BV estimate gives the stated inequality.
\end{proof}

\begin{samepage} 
\begin{proof}[Proof of Proposition~\ref{prop:stability}]
Use the lift in Lemma~\ref{lem:lift} and the density
\(R\) in~\eqref{eq:lifted-density}. Its mass is one and it vanishes
almost everywhere outside \(B\), with
\(V_{(u,0)}(R)=V_u(\rho)\) by translation invariance
and~\eqref{eq:translation}.
Lemma~\ref{lem:rearrangement} gives a jointly measurable probability
density \(R^\star\), zero almost everywhere outside \(B^\star\), with
\begin{equation}\label{eq:variation-contraction}
 V_{(u,0)}(R^\star)\le V_u(\rho)\le\kappa\|u\|_2.
\end{equation} 

The initial uniform marginal on \((-3,3)\) has mean absolute height
\(3/2\). Lemma~\ref{lem:height} and~\eqref{eq:small-step} show that
rearrangement loses at most \(1/2\):
\[
 \int|s|R^\star\ge\frac32-\frac32\kappa\|v\|_2\ge1.
\]
Apply Lemma~\ref{lem:section} to \(B^\star,R^\star\) at height one,
using~\eqref{eq:variation-contraction}.
By~\eqref{eq:section-transform}, this section is \(\Tv_vK\).
It is nonempty and supports a probability density with the same
directional variation bound \(\kappa\).
\end{proof}
\end{samepage}

\begin{remark}[Normalization and the truncation height]\label{rem:height}
For \(v=0\), the identity gives \(\int|s|R^\star=3/2\).
For a one-dimensional equality example, take \(v=1\) and
\(\rho=(2A)^{-1}\mathbf1_{(-A,A)}\), where \(A\ge3\).
Then \(V_1(\rho)=1/A\) and
\(\|\rho-\rho(\cdot+h)\|_1=h/A\) for \(0\le h\le6\).
The mean absolute height is \(3/2-3/(2A)\); at \(A=3\), it equals
one and \(V_1(\rho)=1/3\).

The choice of cutoff three also has a geometric constraint.
Let \(K\) be nonempty, bounded, open, and convex, and take any
\(\rho\in\Prob(K)\). For \(L>1\), put
\(B_L=\{(y,s):|s|<L,\ y+sv\in K\}\), and replace~\eqref{eq:lifted-density} by
\(R_L(y,s)=(2L)^{-1}\rho(y+sv)\mathbf1_{\{|s|<L\}}\).
The same calculation gives
\[
 \int|s|R_L^\star
 =\frac L2-\frac1{8L}\int_0^{2L}(2L-h)
       \|\rho-\rho(\cdot+hv)\|_1\dd h
 \ge\frac L2-\frac{L^2}{6}V_v(\rho).
\]
The height-one section of the symmetrized lift is
\(((K-v)\cap(K+v))+\{tv:|t|<L-1\}\).
Containment in \((K-v)\cup(K+v)\) for every \(K,v\) requires
\(L\le3\): for \(K=(-A,A)\), \(A>1\), and \(v=1\), the section
is \((-A-L+2,A+L-2)\), whereas the union is \((-A-1,A+1)\).
The argument of Lemma~\ref{lem:containment} proves sufficiency.
For \(2<L\le3\), a sufficient condition for the mean absolute height
to be at least one is
\(V_v(\rho)\le3(L-2)/L^2\). The right-hand side increases with \(L\)
and equals \(1/3\) at \(L=3\).
This explains the parameter within the present uniform truncation and
linear translation estimate; it does not assert an optimal stability
threshold for other constructions.
\end{remark}

\appendix
\section{Eigenvalue Convexity for the Regularized Norms}\label{app:eigenvalue}

We give the eigenvalue argument needed in Lemma~\ref{lem:minkowski},
following the infimal-convolution method of
Wang and Xia~\cite[Section~3]{WX11}. The analytic inputs are the
eigenfunction results of Mosconi, Riey, and Squassina~\cite{MRS24}.
The differential comparison is first made away from critical gradients;
a cutoff then gives the global weak inequality, including the critical
set. Throughout this appendix, the norm and \(p>1\) are fixed.

\begin{proposition}\label{prop:eigenvalue-convexity}
Let \(F\) be a norm on \(\R^d\), smooth away from zero, with
\(D^2(F^2)\) positive definite there, and let \(1<p<\infty\).
For nonempty bounded open convex sets \(K_0,K_1\) and \(0\le t\le1\),
the eigenvalues defined by~\eqref{eq:eigenvalue} satisfy
\[
 \lambda_{p,F}((1-t)K_0+tK_1)
 \le(1-t)\lambda_{p,F}(K_0)+t\lambda_{p,F}(K_1).
\]
\end{proposition}
\begin{proof}
The cases \(t=0,1\) are immediate. In dimension one, \(F(\xi)=c|\xi|\)
and the domains are intervals; scaling makes their eigenvalues a positive
constant times the interval length to the power \(-p\), which is convex.
Assume henceforth that \(d\ge2\) and \(0<t<1\). We first take smooth
convex input domains and write
\(K_t=(1-t)K_0+tK_1\) and \(\lambda_i=\lambda_{p,F}(K_i)\).

\emph{Eigenfunction inputs.}
Set \(a(0)=0\) and, for \(\xi\ne0\),
\[
 \begin{aligned}
 a(\xi)&=F(\xi)^{p-1}\nabla F(\xi),\\
 A(\xi)&=Da(\xi)=\frac1pD^2(F^p)(\xi)\\
 &=F(\xi)^{p-2}\bigl((p-1)\nabla F(\xi)\otimes\nabla F(\xi)
                         +F(\xi)D^2F(\xi)\bigr).
 \end{aligned}
\]
Comparison with \(D^2(F^2/2)\) and compactness of the unit sphere give
constants \(c_0,c_1>0\) such that
\begin{equation}\label{eq:ellipticity}
 c_0|\xi|^{p-2}|\zeta|^2
 \le\zeta^{\mathsf T}A(\xi)\zeta
 \le c_1|\xi|^{p-2}|\zeta|^2
 \qquad(\xi\ne0).
\end{equation}
Also \(F^p\in C^1(\R^d)\), since its gradient is
\(O(|\xi|^{p-1})\) at zero.

The direct method gives a nonnegative, nonzero first eigenfunction
\(u_i\in W_0^{1,p}(K_i)\). It minimizes
\[
 w\longmapsto\frac1p\int_{K_i}F(\nabla w)^p
                  -\frac{\lambda_i}{p}\int_{K_i}w^p
 \qquad(w\ge0),
\]
whose minimum is zero by the definition of \(\lambda_i\).
In Theorem~1.1 of~\cite{MRS24}, take the kinetic integrand to be
\(F^p\), the reaction to be \(f(s)=\lambda_i s^{p-1}\), and its
primitive to be \(P(s)=\lambda_i s^p/p\). Then
\(P^{1/p}\) is linear, \(P/f=s/p\) is convex, and the kinetic
integrand is strictly convex. The transformation in that theorem is a
positive multiple of \(\log s\). Thus \(u_i\) is log-concave;
this specialization is also stated explicitly on p.~3674 of~\cite{MRS24}.
The bounds~\eqref{eq:ellipticity} are condition~(33) of that paper.
With the extension \(f(s)=\lambda_i|s|^{p-1}\), Proposition~4.5
of~\cite{MRS24} gives positivity and \(C^{1,\alpha}\) regularity up
to the smooth boundary for this nontrivial minimizer. In particular,
\(u_i\) has bounded gradient and zero boundary value. Where
\(\nabla u_i\ne0\), the equation is locally uniformly elliptic with
smooth coefficients, so local elliptic regularity gives \(u_i\in C^2\).

Consequently \(v_i=-\log u_i\) is convex and \(C^1\) in \(K_i\),
is bounded below, and tends to infinity at the boundary. Away from
\(\{\nabla v_i=0\}\), it is \(C^2\) and satisfies
\begin{equation}\label{eq:log-eigenfunction}
 \operatorname{div}a(\nabla v_i)
 =\lambda_i+(p-1)F(\nabla v_i)^p.
\end{equation}
This identity follows by differentiating \(u_i=e^{-v_i}\) in the
weak eigenfunction equation, which is classical on this region.

\emph{Infimal convolution and its gradients.}
For \(z\in K_t\) and \(\eta\ge0\), define
\[
 w_\eta(z)=\inf_{\substack{x_i\in K_i\\(1-t)x_0+tx_1=z}}
 \bigl((1-t)(v_0(x_0)+\eta|x_0|^2)
                  +t(v_1(x_1)+\eta|x_1|^2)\bigr),
\]
and put \(w=w_0\). The infima are attained at interior pairs, because
the \(v_i\) are bounded below and blow up on their boundaries. When
\(\eta>0\), the pair is unique by strict convexity. Since the input
domains are bounded, \(0\le w_\eta-w\le C\eta\).
For every compact \(Z\subset K_t\), all minimizing pairs with
\(z\in Z\) and \(0\le\eta\le1\) lie in fixed compact subsets of
the input domains: an interior feasible pair remains feasible near
each \(z\) after shifting one component, and a finite cover bounds
the objectives above. The lower bounds on the other component then
bound each \(v_i(x_i)\) above.

Stationarity gives a common gradient
\(q_\eta=\nabla v_i(x_i)+2\eta x_i\). Convexity gives the supporting
lower bound for \(w_\eta(z+h)\), while shifting \(x_0\) by
\(h/(1-t)\) gives the matching first-order upper bound. Hence
\(w_\eta\) is differentiable and \(\nabla w_\eta(z)=q_\eta\),
also at \(\eta=0\), even if the minimizing pair is not unique.
The compactness of minimizing pairs and continuity of \(\nabla v_i\)
now show that \(w_\eta\in C^1(K_t)\) and
\[
 \nabla w_\eta\longrightarrow\nabla w
 \quad\text{uniformly on compact subsets of }K_t.
\]
Indeed, a convergent sequence of output points and perturbation
parameters has a subsequence of minimizing pairs converging to a
minimizer for the limiting problem; its common gradient is uniquely
determined by the output point.

As \(z\) approaches \(\partial K_t\), every limiting decomposition
has a component on its input boundary. Thus \(w(z)\to\infty\).
The critical set \(C=\{\nabla w=0\}\) is therefore a nonempty
compact convex subset of \(K_t\). For \(Z\Subset K_t\setminus C\),
the common gradients stay bounded away from zero. Since
\(\nabla v_i(x_i)=q_\eta-2\eta x_i\), the original gradients also
stay away from zero for small \(\eta\). The corresponding Hessians
\(D^2v_i(x_i)\) are then uniformly bounded on these compact sets.

\emph{The differential inequality away from the critical set.}
At a perturbed minimizing pair, put \(B_i=D^2v_i(x_i)+2\eta I\).
On the compact sets just described, these matrices are positive
definite, and implicit differentiation gives
\[
 D^2w_\eta=\bigl((1-t)B_0^{-1}+tB_1^{-1}\bigr)^{-1}
 \le(1-t)B_0+tB_1.
\]
The inequality is the arithmetic--harmonic inequality for positive
matrices. Taking its trace against \(A(q_\eta)>0\) and
using~\eqref{eq:log-eigenfunction} gives, with
\(\Lambda=(1-t)\lambda_0+t\lambda_1\),
\[
 \operatorname{div}a(\nabla w_\eta)
 \le\Lambda+(p-1)F(\nabla w_\eta)^p+r_\eta,
 \qquad \sup_Z|r_\eta|\longrightarrow0.
\]
To justify the remainder, each original gradient differs from
\(q_\eta\) by \(O(\eta)\), the Hessians are bounded, and \(A\)
is uniformly continuous on the compact annulus containing these
gradients. The added Hessian contributes \(2\eta\operatorname{tr}A\).
The matrix \(A\) has degree \(p-2\); only its bounds on this annulus
are used, also when \(1<p<2\).

Put \(U=e^{-w}\) and \(U_\eta=e^{-w_\eta}\). Evenness and
homogeneity give, at noncritical points,
\[
 \operatorname{div}a(\nabla e^{-v})
 =e^{-(p-1)v}\bigl((p-1)F(\nabla v)^p
                       -\operatorname{div}a(\nabla v)\bigr).
\]
Thus \(\operatorname{div}a(\nabla U_\eta)
\ge-(\Lambda+r_\eta)U_\eta^{p-1}\) locally off \(C\).
Testing and passing to the limit using compact-uniform convergence
of the functions and their first derivatives yields
\begin{equation}\label{eq:weak-off-critical}
 \int_{K_t}a(\nabla U)\cdot\nabla\psi
 \le\Lambda\int_{K_t}U^{p-1}\psi
 \qquad(0\le\psi\in C_c^\infty(K_t\setminus C)).
\end{equation}
No convergence of second derivatives is needed.

\emph{Including the critical set and the boundary.}
The flux \(a(\nabla U)\) is continuous and zero on \(C\).
Choose a Lipschitz cutoff \(\chi_\delta\) which is zero within
distance \(\delta\) of \(C\), one beyond distance \(2\delta\), and
has gradient bounded by \(1/\delta\). Convexity of \(C\) gives
\(\int|\nabla\chi_\delta|\le C_2\), uniformly for small \(\delta\):
the \(2\delta\)-parallel body has volume increment \(O(\delta)\).
This holds whether or not \(C\) has interior. For
\(0\le\phi\in C_c^\infty(K_t)\), insert
\(\psi=\phi\chi_\delta\) in~\eqref{eq:weak-off-critical}, using
smooth approximation away from \(C\). The extra term has magnitude
at most
\[
 C_2\|\phi\|_\infty
 \sup_{\operatorname{dist}(z,C)\le2\delta}|a(\nabla U(z))|
 \longrightarrow0.
\]
Dominated convergence therefore gives
\begin{equation}\label{eq:weak-global}
 \int_{K_t}a(\nabla U)\cdot\nabla\phi
 \le\Lambda\int_{K_t\setminus C}U^{p-1}\phi
 \le\Lambda\int_{K_t}U^{p-1}\phi.
\end{equation}
The first integral includes \(C\) because its flux is zero. A
positive-volume critical set causes no loss, since the right-hand
integrand is nonnegative.

At any minimizing pair the common logarithmic gradient gives
\[
 U(z)=u_0(x_0)^{1-t}u_1(x_1)^t,
 \qquad
 |\nabla U(z)|=|\nabla u_0(x_0)|^{1-t}|\nabla u_1(x_1)|^t.
\]
The bounded gradients of \(u_i\) therefore make \(U\) Lipschitz on
the convex set \(K_t\). If output points tend to its boundary, a
subsequence of minimizing pairs has at least one component on its
input boundary, so the product formula gives \(U\to0\).
Consequently \((U-s)_+\) has compact support in \(K_t\) for every
\(s>0\), and these truncations converge to \(U\) in \(W^{1,p}\)
as \(s\downarrow0\). Mollification proves
\(U\in W_0^{1,p}(K_t)\) and provides nonnegative smooth
approximants. Testing~\eqref{eq:weak-global} with \(U\) now gives
\[
 \int_{K_t}F(\nabla U)^p\le\Lambda\int_{K_t}U^p.
\]
Since \(U>0\), its Rayleigh quotient proves the desired inequality
for smooth input domains.

\emph{Arbitrary convex domains.}
Every bounded open convex domain has nested smooth strongly convex
inner exhaustions. For completeness, translate an interior point to
zero and let \(\gamma\) be its Minkowski functional. Convolution
with a centered nonnegative smooth kernel gives a smooth convex
function at least \(\gamma\), converging uniformly on bounded sets.
Adding \(\delta|x|^2\) makes it strongly convex. Suitable regular
sublevels below one lie compactly inside the domain and contain any
prescribed compact subset. Successively including the preceding
closure gives the required nested exhaustion.

Apply this to both inputs, obtaining \(K_{i,j}\uparrow K_i\).
Their Minkowski combinations increase to \(K_t\). For every open
exhaustion \(L_j\uparrow L\), domain monotonicity and compactly
supported smooth Rayleigh tests give
\(\lambda_{p,F}(L_j)\downarrow\lambda_{p,F}(L)\).
Passing to this limit proves the assertion.
\end{proof}

All regularity and perturbation constants in this proof may depend on
the fixed norm, \(p\), the dimension, and the auxiliary domains.
Lemma~\ref{lem:minkowski} subsequently takes \(p\downarrow1\) and
\(\epsilon\downarrow0\) using scalar infima and fixed test densities;
it requires no uniform eigenfunction estimates in either limit.

\bibliographystyle{alpha}
\bibliography{komlos-variation}
\end{document}